\documentclass[12pt]{amsart}

\usepackage{amsmath}
\usepackage{amssymb}
\usepackage{mathdots}
\usepackage{calrsfs}
\usepackage[autostyle]{csquotes}
\usepackage{float}
\usepackage{tikz-cd} 
\usetikzlibrary{patterns}
\usetikzlibrary{matrix}
\usepackage{tikz}
\usepackage{amssymb,latexsym,amsmath,amsthm,amsfonts,graphicx}
\usepackage{verbatim}
\usepackage[autostyle]{csquotes}
\usepackage{mathrsfs}
\usepackage{float}
\usepackage{pst-node}
\usepackage{mathtools}
\usepackage{faktor}
\usepackage{hyperref}
\usepackage{orcidlink}
\usepackage{cancel}
\usepackage{faktor}
\usepackage{orcidlink}

\tikzstyle{littledot}=[circle, fill, inner sep=.4pt,minimum size=.4pt]
\tikzstyle{vertex}=[circle, fill, inner sep=2pt]

\newtheorem{thm}{Theorem}[section]
\newtheorem{lem}[thm]{Lemma}

\newtheorem{prop}[thm]{Proposition}
\newtheorem{cor}[thm]{Corollary}

\newtheorem{defn}[thm]{Definition}

\newtheorem{remark}[thm]{Remark}
\newtheorem{lemdef}[thm]{Definition/Lemma}

\newcommand{\N}{{\mathbb N}}

\newcommand{\Z}{{\mathbb Z}}

\newcommand{\VD}{\mathit{VD}}

\newcommand{\fin}{\sim_\textrm{F}}
\newcommand{\timesf}{\cdot_\textrm{F}}

\newcommand{\fc}{\boldsymbol{c}_\mathit{F}}
\newcommand{\RON}{R^{\mathbf{ON}}}
\newcommand{\ot}{\operatorname{o.\!t.}}
\newcommand{\ON}{\mathbf{ON}}
\newcommand{\bsd}{\partial_{F}} 
\newcommand{\bsi}{\partial_{F}^\star} 
\newcommand{\fd}{\omega[\omega]^\omega_{CNF}}
\newcommand{\mon}{\operatorname{Mon}_\omega} 

\newcommand{\End}{\operatorname{End}}

\title[Algebraic structures arising from the finite condensation]{Algebraic structures arising from the finite condensation on linear orders}

\author{Jennifer Brown and Ricardo Su\'{a}rez}
\address{Department of Mathematics, California State University Channel Islands, Camarillo, CA, USA}
\email{jennifer.brown@csuci.edu,  ricardo.suarez532@csuci.edu}

\begin{document}

\maketitle

\begin{abstract}
The finite condensation $\sim_F$ is an equivalence relation defined on a linear order $L$ by $x \sim_F y$ if and only if the set of points lying between $x$ and $y$ is finite. We define an operation $\cdot_F$ on linear orders $L$ and $M$ by $L \cdot_F M = \operatorname{o.t.}\left((LM)/\!\sim_F\right)$; that is, $L \cdot_F M$ is the order type of the lexicographic product of $L$ and $M$ modulo the finite condensation. The infinite order types $L$ such that $L / \! \sim_F\, \cong 1$ are $\omega, \omega^*,$ and  $\zeta$ (where $\omega^*$ is the reverse ordering of $\omega$, and $\zeta$ is the order type of $\mathbb{Z}$). We show that under the operation $\cdot_F$, the set $R=\{1, \omega, \omega^*, \zeta\}$ forms a left regular band. Further, each of the ordinal elements of $R$ defines, via left or right multiplication modulo the finite condensation, a weakly order-preserving map on the class of ordinals. We study these maps' effect on the ordinals of finite degree in Cantor normal form. In particular, we examine the extent to which one of these maps, sending $\alpha$ to $1 \timesf \alpha \cong \ot(\faktor{\alpha}{\sim_F})$, behaves similarly to a derivative operator on the ordinals of finite degree in Cantor normal form. 
\end{abstract}

\section{Introduction}

For $L$ a linear order, define an equivalence relation on $L$ by $x \fin y$ if there are only finitely many points between $x$ and $y$. The \textit{finite condensation} of $L$ is the partition of $L$ into equivalence classes associated with this relation. The resulting equivalence classes are intervals of $L$, and the collection $\faktor{L}{\fin}$ of these intervals forms a linear order itself, with the order inherited from $L$. The finite condensation has been used extensively in the literature on linear orders: for example, to show that any countably infinite linear order is order-isomorphic to a proper subset of itself \cite{DuMi}; and, more recently, to classify the countable linear orders that are left-absorbing under (lexicographic) linear order multiplication \cite{ErGu}.

It is straightforward to verify that $\faktor{\N}{\fin} \cong 1$; $\faktor{\N^*}{\fin} \cong 1$ (where $\N^*$ denotes $\N$ with the reverse order); $\faktor{\Z}{\fin} \cong 1$; and $\faktor{n}{\fin} \cong 1$ for any natural number $n$ (considered as a linear order). Conversely, if $L$ is any linear order such that $\faktor{L}{\fin} \cong 1$, then $L$ must be order-isomorphic to either $\N$, $\N^*$, $\Z$, or some finite $n$. We denote the order types of $\N$, $\N^*$, and $\Z$ by $\omega, \omega^*$, and $\zeta$ respectively. 

If $A$ is a collection of order types, the \textit{ring of types} generated by $A$, denoted $[A]$, is the closure of $A$ under sums and generalized sums: so $L_1 + L_2 \in [A]$ whenever $L_1, L_2 \in [A]$, and $\sum_{i \in I} L_i \in [A]$ whenever $I$ and each $L_i$ are in $[A]$. Hausdorff proved in \cite{Ha} that the ring of types generated by the set $\{\omega, \omega^*, 1\}$ is equal to the set of countable scattered linear orders. Including the order type $\zeta$ in this set to obtain $\{1, \omega, \omega^*, \zeta\}$, we have -- except for the finite linear orders of size greater than 1 -- the order types whose finite condensation is equal to $1$. ($\zeta$ is also a natural order type to include in $R$ because $\omega$, $\omega^*$, and $\zeta$ are exactly the infinite order types used to define the hierarchy of countable very discrete sets, used in proving Cantor's Theorem.) In Section \ref{s: left regular band}, we introduce an operation defined in terms of the finite condensation: $L \timesf M$ is the order type of the lexicographic product of $L$ and $M$, modulo the finite condensation. Under this operation, $(R, \timesf)$ forms a \textit{left regular band} (see Theorem \ref{R is a left regular band}).

The set $\RON := \{1, \omega\}$ of ordinal elements of $R$ also forms a left regular band. In Section \ref{s: endomorphisms}, we use the product $\timesf$ to define left and right actions $\phi^F_l$ and $\phi^F_r$ from $\RON$ to the class of weakly order-preserving maps on $\ON$: $\phi^F_l(1)$ and $\phi^F_r(1)$ send an ordinal $\alpha$ to the order type of its finite condensation; $\phi^F_l(\omega)$ left multiplies $\alpha$ by $\omega$, applies the finite condensation, and then takes the order type of the resulting linear order; and $\phi^F_r(\omega)$ is the identity map. We examine the effects of these weakly order-preserving maps, which we term \textit{endomorphisms of} $\ON$, on the set of ordinals whose Cantor normal form is of finite degree. In Section \ref{s: the phi maps on ordinals of finite degree}, we describe these maps' action on the ordinals whose Cantor normal form is of finite degree. 

For $\alpha$ an ordinal of finite degree, the map $\alpha \mapsto \faktor{\alpha}{\fin}$ sends $\alpha$ to its finite condensation, which is not in general an ordinal (though it is isomorphic to one). By contrast, the map $\phi^F_l(1)$, which sends an ordinal $\alpha$ to $\ot(\faktor{1 \cdot \alpha}{\fin}) \cong \ot(\faktor{\alpha}{\fin})$, is really a map from the set of ordinals of finite degree to itself. In some important respects, $\phi^F_l(1)$ acts like a derivative operator on the set of ordinals of finite degree in Cantor normal form. In Section \ref{s: finite condensation derivatives}, we assign $\phi^F_l(1)$ the suggestive notation $\bsd$, and examine the extent to which the map $\bsd$ acts like a (linear) derivative operator. By restricting the domain to limit ordinals of finite degree, we are able to define an inverse function $\bsi$ to $\bsd$ in Proposition \ref{bsi as an inverse}. We show in Theorem \ref{bsd is linear on limit ordinals of finite degree at least 2}  that $\bsd(p \alpha + q \beta) = p \bsd(\alpha) + q \bsd(\beta)$ when $\alpha$ and $\beta$ are ordinals of finite degree at least $2$ in Cantor normal form.

\section{Basic notions}

In order that this paper be relatively self-contained, we first review some basic definitions. For proofs, and for background on linear orders and condensations, see \cite{Ro}.

\begin{defn}\label{d: linear order}
A \textbf{linear order} is a set $L$ together with a binary relation $\leq_L$ that is a partial order under which any two elements of $L$ are comparable. That is, for all $x, y, z \in L$,
\begin{enumerate}
    \item $x \leq_L x$ (reflexivity),
    \item $[(x \leq_L y) \wedge (y \leq_L x)] \implies x=y$ (antisymmetry), and 
    \item $[(x \leq_L y) \wedge (y \leq_L z)] \implies x \leq_L z$ (transitivity); and, in addition,
    \item $(x \leq_L y) \vee (y \leq_L x)$ (any two elements are comparable).
\end{enumerate}
A linear order in which every non-empty subset has a least element is a \textbf{well-order}. Linear orders $(L, \leq_L)$ and $(M, \leq_M)$ are \textbf{order-isomorphic}, or have the same \textbf{order type}, if there is an order-preserving bijection from $L$ to $M$. If there is an order-preserving injection from $L$ to $M$, we say that $L$ \textbf{embeds} into $M$, and write $L \preceq M$.
\end{defn}
 
We will drop the subscripts from our $\leq$ signs where the context is clear. We will use the convention in \cite{Ro} that we have selected, from each class of order types, a particular representative; and, in the case of well-orders, we will choose that order type to be an ordinal (every well-order is isomorphic to an ordinal). For convenience, we will also take $\Z$ and $\N^*$ to be the designated representatives of their respective equivalence classes. $\ot(L)$ will denote the order type of a linear order $L$.

An \textit{interval} in a linear order $L$ is a subset $I$ of $L$ such that whenever $x, y \in I$ and $z \in L$ with $x \leq z \leq y$, we have $z \in I$. That is, an interval is a convex subset of $L$. We use the notation $[\{x, y\}]$, where $x, y \in L$, to denote either $[x,y]$ (if $x \leq y$) or $[y,x]$ (if $y < x$).

\begin{lemdef}\label{d: condensation}
A \textbf{condensation} of a linear order $L$ is an equivalence relation $\sim$ on $L$ whose equivalence classes are intervals. A condensation $\sim$ is uniquely associated with a map $\boldsymbol{c}: L \to \faktor{L}{\sim}$, called the \textbf{condensation map}, sending each $x \in L$ to its equivalence class mod $\sim$; that is, $\boldsymbol{c}(x) = \{y \in L: y \sim x\}$. The set of equivalence classes, $\faktor{L}{\sim} = \{\boldsymbol{c}(x): x \in L\}$, is a linear order under the order inherited from $L$: if $I_1$ and $I_2$ are elements of $\faktor{L}{\sim}$, then $I_1 < I_2$ iff $x_1 < x_2$ for each $x_1 \in I_1$ and $x_2 \in I_2$. The condensation map $\boldsymbol{c}$ is an order-homomorphism from $L$ to $\faktor{L}{\fin}$. 
\hfill \qed
\end{lemdef}

The induced order on $\faktor{L}{\fin}$ is well-defined since equivalence classes mod $\sim$ are intervals in $L$. We have by \ref{d: condensation} that if $x, y \in L$ with $x\leq y$, then $\boldsymbol{c}(x) \leq \boldsymbol{c}(y)$ in $\faktor{L}{\fin}$. We will sometimes denote $\faktor{L}{\fin}$ by $\boldsymbol{c}[L]$.

In the next section, we will define a multiplication on linear orders based on the \textit{finite condensation}. Under the finite condensation, elements of a linear order are declared to be equivalent when there are only finitely many points between them. 

\begin{lemdef}\label{finite condensation really is one}
Let $L$ be any linear order. For $x, y \in L$, say that $x \fin y$ iff $[\{x, y\}]$ is finite. The relation $\fin$ is a condensation on $L$, called the \textbf{finite condensation}. \hfill \qed
\end{lemdef}

$\fc$ will denote the finite condensation map. 

We define addition of linear orders in the usual way (see \cite{Ku} or \cite{Ro}): intuitively, if $L$ and $M$ are linear orders, $L+M$ is obtained by laying out a copy of $L$ followed by a copy of $M$. 

\begin{defn}\label{d: linear order addition}
Let $L$ and $M$ be linear orders. Define the \textbf{sum} $L + M$ as $((L \times \{0\}) \cup (M \times \{1\}),R)$, where 
\begin{align*}
R= & \{\langle \langle l,0 \rangle, \langle l', 0 \rangle \rangle : l < l' \textrm{ in } L\}  \\
 & \quad \quad \cup \{\langle \langle m, 1 \rangle, \langle m', 1 \rangle \rangle: m < m' \textrm{ in } M\} \\ & \quad \quad \quad \quad \cup [(L \times \{0\}) \times (M \times \{1\})].
\end{align*}   
\end{defn}

For multiplication of linear orders, we use the \textit{lexicographic} ordering, as in \cite{Er} or \cite{Si}, rather than the antilexicographic ordering used in \cite{Ku} or \cite{Ro}. We define the product $LM$ of linear orders $L$ and $M$ as the linear order obtained by putting the lexicographic order on $L \times M$. Intuitively, to form $LM$, one replaces each $l \in L$ with a copy of $M$. (For example, under this lexicographic product, $2 \omega \cong \omega + \omega$, but $\omega 2 \cong \omega$.)

\begin{defn}\label{d: linear order multiplication}
Suppose $L$ and $M$ are linear orders. Then $LM$ is the  lexicographic ordering $R$ of the Cartesian product $L \times M$:
\[
\langle l, m \rangle \, R \, \langle l', m' \rangle \iff (l < l' \textrm{ or } (l = l' \textrm{ and } m < m'))
\]
for $l, l' \in L$ and $m, m' \in M$. 
\end{defn}


The following facts about monotonicity of linear order addition and (lexicographic order) multiplication on ordinals will often be useful (see, for example, \cite{De}):

\begin{lem}\label{monotonicity}
Let $\alpha, \beta,$ and $\gamma$ be ordinals, and define addition and (lexicographic) multiplication as above. Then:
\begin{enumerate}
    \item Strict right monotonicity of $+$: if $\beta < \gamma$, then $\alpha + \beta < \alpha + \gamma$;
    \item Weak left monotonicity of $+$: if $\beta < \gamma$, then $\beta + \alpha \leq \gamma + \alpha$;
    \item Strict left monotonicity of lexicographic order multiplication: if $\beta < \gamma$, then $\beta \alpha < \gamma \alpha$; and 
    \item Weak right monotonicity of lexicographic order multiplication: if $\beta < \gamma$, then $\alpha \beta \leq \alpha \gamma$.
    \hfill \qed
\end{enumerate}
    
\end{lem}

\section{The left regular band $R$ and the product modulo the finite condensation}\label{s: left regular band}

In this section, we show that a set of four order types arising naturally from Hausdorff's Theorem forms a left regular band (first defined in \cite{Cl}) under an operation defined in terms of the finite condensation. A \textit{scattered} linear order is one that does not contain a copy of the rationals $\mathbb{Q}$. Hausdorff's theorem on the countable scattered linear orderings (\cite{Ha}; see \cite{Ro}) states that a countable linear ordering $L$ is scattered if and only if it is very discrete. Loosely speaking, any countable scattered linear order can be built recursively as a generalized sum from a ``basis'' consisting of the simple order types $\omega$, $\omega^*$, and $\zeta$, along with finite order types. Here $\omega$ denotes the order type of the natural numbers $\mathbb{N}$; $\omega^*$ denotes the order type of the negative integers $\N^*$; and $\zeta$ denotes the order type of the integers $\mathbb{Z}$. 

\begin{defn}\label{d: very discrete}
The classes $\VD_\alpha$ of linear orders, for each countable ordinal $\alpha < \omega_1$, are defined inductively as follows:
\begin{enumerate}
    \item $0$ and $1$ are in $\VD_0$.
    \item Let $\alpha < \omega_1$ and suppose $\VD_\beta$ has already been defined for each $\beta < \alpha$. Then let  $\VD_\alpha$ consist of all linear orders of the form $
    \sum_{i \in I}L_i$
    where $I$ is among $\omega, \omega^*, \zeta$, or $n$ (for some $n \in \omega$), and where for each $i \in I$, $L_i \in \bigcup\{\VD_\beta: \beta < \alpha\}$.
\end{enumerate}
The class $\VD$ of \textbf{(countable) very discrete linear orders} is then defined to be the union of the classes $\VD_\alpha$; that is, $\VD = \bigcup_{\alpha < \omega_1} \VD_\alpha$.
\end{defn}

\begin{thm}[Hausdorff]\label{Hausdorff's theorem}
A countable linear ordering $L$ is scattered if and only if it is very discrete. 
\end{thm}

We define $R:=\{1, \omega, \omega^*, \zeta\}$. The set $R$, then, is  similar to the ``basis'' used to recursively build the countable scattered linear orders. For any linear order $L$ and for any $x \in L$, the condensation class $c_F(x)$ is either finite or isomorphic to one of the order types $\omega$, $\omega^*$, or $\zeta$ (exercise 4.5 in \cite{Ro}). On the other hand, it is clear that any linear order isomorphic to an element of $R$, when modded-out by the finite condensation, is isomorphic to the 1-element linear order, because $\faktor{\N}{\fin} \cong \faktor{\N^*}{\fin} \cong \faktor{\Z}{\fin} \cong \faktor{1}{\fin} \cong 1$. Elements of $R$ also have the following property (when the linear order $X$ in the lemma is $\omega$, this is part of Exercise 4.5 in \cite{Ro}):

\begin{lem}\label{right absorption mod fin in R}
Let $M$ be any linear order and let $X$ be a linear order of order type $\omega, \omega^*$, or $\zeta$. Then $\faktor{MX}{\fin} \cong M$. 
\end{lem}

\begin{proof}
First observe that if the order type of $X$ is among $\omega, \omega^*$, or $\zeta$, and if $x$ is any point in $X$, then $\fc(x) = X$. 

Let $M$ be any linear order; let $X$ be order-isomorphic to one of $\omega, \omega^*$, or $\zeta$; and consider a point $r \in MX$. Then for some $m \in M$, $r$ is a point in the copy $X_m$ of $X$ that replaced $m$ in forming the product $MX$. By the above observation, $\fc(r) \supseteq X_m$. Now take any $s \in MX$ that is in a copy $X_{m'}$ of $X$ for some $m' \neq m$ in $M$. By construction, there are infinitely many points between $r$ and $s$, so $s \cancel{\sim}_\textrm{F} r$. Then $\fc(r) \subseteq X_m$, and so we have $\fc(r) = X_m$. That is, each copy $X_m$ of $X$, for $m \in M$, is its own equivalence class mod $\fin$. Therefore $\faktor{MX}{\fin} \cong M$.
\end{proof}

We define a multiplication modulo the finite condensation. 

\begin{defn}\label{d: timesf}
For linear orders $L$ and $M$, define an operation $\timesf$ (\textbf{multiplication modulo the finite condensation}) by $L \timesf M := \ot(\faktor{L M}{\fin})$. 
\end{defn}

Note that if $L$ and $M$ are ordinals, then $L \timesf M$ is also an ordinal. 

Suppose $L$ and $M$ are infinite linear orders. In forming $\faktor{LM}{\fin}$ from $LM$, each copy $M_l$ of $M$ gets replaced by $\faktor{M_l}{\fin}$.   

\begin{lem}\label{timesf is associative on R}
For all $X, Y, Z \in R$, 
\[
(X \timesf Y) \timesf Z \cong X \timesf (Y \timesf Z).
\]
That is, $\timesf$ is an associative operation on $R$.
\end{lem}

\begin{proof}
We proceed by cases. Clearly both sides of the equation are isomorphic to $1$ if $X, Y, $ and $Z$ are all $1$, so we may assume that at least one of $X, Y$, and $Z$ is in $\{\omega, \omega^*, \zeta\}$.

\textit{Case 1:} First suppose $Z \in \{\omega, \omega^*, \zeta\}$. Then by Lemma \ref{right absorption mod fin in R}, we have that for any $X, Y \in R$, 
\[
(X \timesf Y) \timesf Z \cong X \timesf Y \cong X \timesf (Y \timesf Z).
\]
\textit{Case 2a:} Suppose $Z=1$ and $Y \in \{\omega, \omega^*, \zeta\}$. Then by Lemma \ref{right absorption mod fin in R} and because $\faktor{X}{\fin} \cong 1$ for every $X \in R$, we have 
\[
(X \timesf Y) \timesf Z \cong X \timesf Z = X \timesf 1 = \ot(\faktor{X1}{\fin}) \cong \ot(\faktor{X}{\fin}) \cong 1.
\]
On the other hand, 
\[
X \timesf (Y \timesf Z) \cong X \timesf (\ot(\faktor{Y1}{\fin})) \cong X \timesf (\ot(\faktor{Y}{\fin})) \cong X \timesf 1 \cong \faktor{X}{\fin} \cong 1.
\]
\textit{Case 2b:} Finally, suppose $Z=1$ and $Y=1$. Recall that we are supposing that not all of $X, Y$, and $Z$ are finite, so we may assume that $X \in \{\omega, \omega^*, \zeta\}$. Then
\[
X \timesf (Y \timesf Z) \cong X \timesf (1 \timesf 1) \cong X \timesf 1 \cong \faktor{X}{\fin} \cong 1.
\]
Also, 
\[
(X \timesf Y) \timesf Z \cong (\faktor{X1}{\fin}) \timesf 1 \cong (\faktor{X}{\fin}) \timesf 1 \cong 1 \timesf 1 \cong 1.
\]
\end{proof}


\begin{lem}\label{timesf multiplication table}
We have the following multiplication table for $R$ under $\timesf$:
\begin{center}
\begin{tabular}{l||l|l|l|l}
$\timesf$ & $1$ & $\omega$ & $\omega^*$ & $\zeta$ \\ \hline \hline
$1$ & $1$ & $1$ & $1$ & $1$ \\ \hline
$\omega$ & $1$ & $\omega$ & $\omega$ & $\omega$ \\ \hline
$\omega^*$ & $1$ & $\omega^*$ & $\omega^*$ & $\omega^*$ \\ \hline
$\zeta$ & $1$ & $\zeta$ & $\zeta$ & $\zeta$ \\ 
\end{tabular}
\end{center}
\end{lem}

\begin{proof}
The entries in the first row and first column of the table are $1$ since $1L \cong L1 \cong L$ for any linear order $L$, and since $\faktor{L}{\fin} \cong 1$ if $L$ has order type $\omega, \omega^*$, or $\zeta$. The other entries follow from Lemma \ref{right absorption mod fin in R}.
\end{proof}

A \textit{semigroup} is a set with an associative binary operation. A \textit{band} is a semigroup in which every element is idempotent; and a \textit{left-regular band} is a band $B$ such that $xyx=xy$ for all $x, y \in B$. 

\begin{thm}\label{R is a left regular band}
$(R, \timesf)$ is a left regular band.
\end{thm}
\begin{proof}
We know by Lemma \ref{timesf is associative on R} that $\timesf$ is an associative operation, and $R$ is closed under $\timesf$ by Lemma \ref{timesf multiplication table}, so $(R, \timesf)$ is a semigroup. We see from the multiplication table in Lemma \ref{timesf multiplication table} that every element of $R$ is an idempotent element, so $(R, \timesf)$ is a band.  

Finally, let $X, Y \in R$; we show that $X \timesf Y \timesf X \cong X \timesf Y$. First, observe from the multiplication table in Lemma \ref{timesf multiplication table} that $1$ is an absorbing element: for all $X \in R$, $X \timesf 1 = 1 = 1 \timesf X$ for all $X \in R$. Therefore $X \timesf Y \timesf X \cong X \timesf Y \cong 1$ whenever one of $X, Y$ is $1$. Now suppose both $X$ and $Y$ are among $\omega, \omega^*$, or $\zeta$. Then $X \timesf Y \timesf X \cong X \timesf Y$ by Lemma \ref{right absorption mod fin in R}. Thus $R$ is a left regular band. 
\end{proof}

Since $(R, \timesf)$ is a left regular band, the relation $\leq_{\timesf}$ on $R$, defined by letting $X \leq_{\timesf} Y$ exactly when $X \timesf Y \cong Y$, is a partial order. From the multiplication table for $\timesf$ in Lemma \ref{timesf multiplication table}, we obtain the following Hasse diagram for the resulting poset:
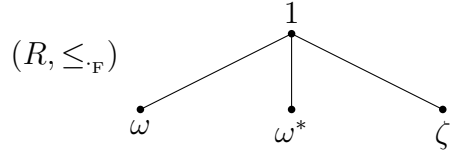
\[
\begin{tikzpicture}
\tikzstyle{vertex}=[circle, fill, inner sep=1pt]   
\draw (-2,0) -- (0,1) -- (2,0);
\draw (0,0) -- (0,1);
\node [below] at (-2,0) {$\omega$};
\node [below] at (0,0) {$\omega^*$};
\node [below] at (2,0) {$\zeta$};
\node [above] at (0,1) {$1$};
\node [vertex] at (-2,0) {};
\node [vertex] at (0,0) {};
\node [vertex] at (2,0) {};
\node [vertex] at (0,1) {};
\node at (-3, .7) {$(R, \leq_{\timesf})$};
\end{tikzpicture}
\]

If we consider the operation $\timesf$ on just the ordinal elements $1$ and $\omega$ of $R$, we also get a left regular band.

\begin{cor}
$\RON=\{1, \omega\}$ is a left regular band under the operation $\timesf$.
\end{cor}

\begin{proof}
We can see by the multiplication table in Lemma \ref{timesf multiplication table} that $\RON$ is closed under $\timesf$. Since $R$ is a left regular band, the other properties of a left regular band are inherited from $R$.
\end{proof}

\section{Endomorphisms defined by multiplication mod the finite condensation}\label{s: endomorphisms}

For $X \in \RON$, define a class function $\phi^F_r(X)$ on the class $\mathbf{ON}$ of ordinals by $\phi^F_r(X)(\alpha) = \alpha \timesf X$, for $\alpha$ an ordinal.  Similarly, define $\phi^F_l(X)(L) = X \timesf L$. We will show in this section that each of the maps $\phi^F_r(1)$, $\phi^F_l(1)$, $\phi^F_r(\omega)$, and $\phi^F_l(\omega)$ is what we will refer to as an \textit{endomorphism of $\mathbf{ON}$}, in the sense of a weakly order-preserving map from $\mathbf{ON}$ to $\mathbf{ON}$.

First, $\phi^F_r(\omega)$ is certainly an endomorphism of $\ON$, as it is the identity map when applied to an ordinal:

\begin{lem}\label{timesf by omega on right is the identity}
If $\alpha$ and $\beta$ are ordinals with $\alpha < \beta$, then $\phi^F_r(\omega)(\alpha) =\alpha < \beta =\phi^F_r(\omega)(\beta)$.
\end{lem}

\begin{proof}
Let $\alpha$ and $\beta$ be ordinals with $\alpha < \beta$. Then by Lemma \ref{right absorption mod fin in R},
\begin{align*}
\phi^F_r(\omega)(\alpha) &= \alpha \timesf \omega = \ot(\faktor{\alpha \omega}{\fin}) \cong \ot(\alpha) = \alpha \\
& < \beta = \ot(\beta) \cong \ot(\faktor{\beta \omega}{\fin}) = \beta \timesf \omega = \phi^F_r(\omega)(\beta).
\end{align*}
\end{proof}

In order to show that the maps $\phi^F_r(1)$ and $\phi^F_l(1)$ are endomorphisms of $\ON$, we first verify that the map $L \mapsto \faktor{L}{\fin}$ sends well-orders to well-orders.

\begin{lem}\label{well ordered-ness inherited}
If $L$ is well-ordered, then so is $\faktor{L}{\fin}$.
\end{lem}

\begin{proof}
The proof is by contrapositive. Suppose $L$ is a linear order such that $\faktor{L}{\fin}$ is not well-ordered. Then we can find an infinite descending sequence of elements of $\faktor{L}{\fin}$. That is, there are $x_n \in L$, for $n \in \omega$, such that $c_F(x_0) > c_F(x_1) > \cdots > c_F(x_n) > c_F(x_{n+1}) > \cdots$ in $\faktor{L}{\fin}$. (Recall that elements of $\faktor{L}{\fin}$ are finite condensation classes of elements of $L$.) But then also $x_0 > x_1 > \cdots > x_n > x_{n+1} > \cdots$ in $L$, so that $L$ is not well-ordered either.
\end{proof}

\begin{lem}\label{finite condensation is weakly order-preserving}
If $\alpha$ and $\beta$ are ordinals with $\alpha < \beta$, then $\phi^F_l(1)(\alpha) = \phi^F_r(1)(\alpha) \leq \phi^F_r(1)(\beta) = \phi^F_l(1)(\alpha)$.
\end{lem}

\begin{proof}
Note that $\phi^F_r(1)$ acts identically to $\phi^F_l(1)$: for any ordinal $\alpha$, $\phi^F_l(1)(\alpha) = \phi^F_r(1)(\alpha) = \ot(\faktor{\alpha}{\fin})$, since $1\alpha \cong \alpha 1$. That is, both $\phi^F_l(1)$ and $\phi^F_r(1)$ send an ordinal $\alpha$ to the ordinal isomorphic to its finite condensation.

We know by Lemma \ref{well ordered-ness inherited} that the finite condensation of a well-ordering is also a well-ordering. Suppose $\alpha$ and $\beta$ are ordinals with $\alpha < \beta$. Then $\alpha \subseteq \beta$. Note that the equivalence class mod $\fin$ of an $x \in \alpha$ might be different depending on whether we take $\fc(x)$ with respect to $\alpha$ or with respect to $\beta$. For this reason, for each $x \in \alpha$, denote $^\alpha \fc(x) = \{y \in \alpha: y \fin x\}$, and denote $^\beta \fc(x) = \{y \in \beta: y \fin x\}$. Let $x \in \alpha$. If $^\beta \fc(x)$ is entirely contained in $\alpha$, then $^\alpha \fc(x) =\, ^\beta \fc(x)$. If $^\beta \fc(x) \not\subseteq \alpha$ -- that is, if $^\beta \fc(x)$ overlaps with both $\alpha$ and $\beta \setminus \alpha$ -- then $^\alpha \fc(x) =\, ^\beta \fc(x) \cap \alpha$. In fact, $^\alpha \fc(x) =\, ^\beta \fc(x) \cap \alpha$ in any case. Define a map $g:\faktor{\alpha}{\fin} \to \faktor{\beta}{\fin}$ by $g(^\alpha \fc(x)) =\, ^\beta \fc(x)$, for $x \in \alpha$. Then $g$ is injective: for if $^\beta \fc(x)=\,^\beta \fc(y)$ for some $x, y \in \alpha$, then $^\beta \fc(x) \cap \alpha=\,^\beta \fc(y) \cap \alpha$, so $^\alpha \fc(x) = \, ^\alpha \fc(y)$. In fact, $g$ sends each interval $^\alpha \fc(x)$ in $\faktor{\alpha}{\fin}$ to its exact copy in $\beta$, except possibly for the very last one in case $\alpha$ is a successor ordinal. (In that case, that very last interval $^\alpha \fc(x)$ gets sent by $g$ to something a little bigger -- containing some elements of $\beta \setminus \alpha$.) This means that $g$ maps $\faktor{\alpha}{\fin}$ order-isomorphically to an initial segment of $\faktor{\beta}{\fin}$. Therefore $\phi^F_r(1)(\alpha) = \ot(\faktor{\alpha}{\fin}) \leq \ot(\faktor{\beta}{\fin}) = \phi^F_r(1)(\beta)$, as desired.
\end{proof}

We next show that left multiplication by $\omega$ modulo the finite condensation is an endomorphism of $\ON$.

\begin{cor}\label{phi left by omega is an endomorphism}
If $\alpha$ and $\beta$ are ordinals with $\alpha < \beta$, then $\phi^F_l(\omega)(\alpha) \leq \phi^F_l(\omega)(\beta)$. That is, $\phi^F_l(\omega)$ is a weakly order-preserving map on $\mathbf{ON}$. 
\end{cor}

\begin{proof}
Let $\alpha$ and $\beta$ be ordinals with $\alpha < \beta$.  Then $\omega \alpha \leq \omega \beta$, since ordinal multiplication is weakly monotone. Then by Lemma \ref{finite condensation is weakly order-preserving}, we have 
\[
\phi^F_l(\omega)(\alpha) = \omega \timesf \alpha = \ot(\faktor{\omega \alpha}{\fin}) \leq \ot(\faktor{\omega \beta}{\fin}) = \omega \timesf \beta = \phi^F_l(\omega)(\beta).
\]
\end{proof}

Therefore we have, by Lemmas \ref{timesf by omega on right is the identity}, \ref{finite condensation is weakly order-preserving}, and \ref{phi left by omega is an endomorphism}, that each of the class maps $\phi^F_r(\omega), \phi^F_l(1), \phi^F_r(1)$, and $\phi^F_l(\omega)$ is an endomorphism of $\ON$. Stated another way, we have class maps $\phi^F_l:\RON \to \End(\ON)$ and $\phi^F_r:\RON \to \End(\ON)$. One might ask whether $\phi^F_l$ and $\phi^F_r$ act like true representations in the sense of structure-preserving maps from the left regular band $\RON$ under $\timesf$ to the class $\End(\ON)$ under composition. That is, given $X, Y \in \RON$, is it the case that $\phi^F_r(X \timesf Y) = \phi^F_r(X) \circ \phi^F_r(Y)$? The answer is ``no''; $\phi^F_r$ preserves the products $\omega \timesf 1$, $1 \timesf \omega$, and $\omega \timesf \omega$, but it does not preserve the product $1 \timesf 1$. (A similar situation holds for the map $\phi^F_l$.) We check these four cases, using Lemma \ref{right absorption mod fin in R} repeatedly, along with the fact that $\omega \timesf 1 = \ot(\faktor{\omega 1}{\fin}) = \ot(\faktor{\omega}{\fin}) = 1$. Let $\alpha$ be any ordinal. 

\textbf{Case 1:} $X=Y=\omega$. Then 
\[
\phi^F_r(\omega \timesf \omega)(\alpha) = \phi^F_r(\omega)(\alpha) = \alpha \timesf \omega = \alpha,
\]
and also 
\begin{align*}
[\phi^F_r(\omega) \circ \phi^F_r(\omega)](\alpha) &= \phi^F_r(\omega)(\phi^F_r(\omega)(\alpha)) = \phi^F_r(\omega)(\alpha \timesf \omega) \\ &= \phi^F_r(\omega)(\alpha) = \alpha \timesf \omega = \alpha.
\end{align*}

\textbf{Case 2:} $X=\omega$ and $Y=1$. Then 
\[
\phi^F_r(\omega \timesf 1)(\alpha) = \phi^F_r(1)(\alpha) = \alpha \timesf 1 = \ot(\faktor{\alpha}{\fin}),
\]
and also
\begin{align*}
[\phi^F_r(\omega) \circ \phi^F_r(1)](\alpha) & = \phi^F_r(\omega)(\phi^F_r(1)(\alpha)) = \phi^F_r(\omega)(\alpha \timesf 1) \\ &= \phi^F_r(\omega)(\ot(\faktor{\alpha}{\fin})) = (\ot(\faktor{\alpha}{\fin})) \timesf \omega \\
& =  \ot(\faktor{\alpha}{\fin}).
\end{align*}

\textbf{Case 3:} $X=1$ and $Y=\omega$: this is similar to Case 2.

\textbf{Case 4:} $X=Y=1$. Then 
\[
\phi^F_r(1 \timesf 1)(\alpha) = \phi^F_r(1)(\alpha) = \alpha \timesf 1 = \ot(\faktor{\alpha}{\fin}),
\]
but
\begin{align*}
[\phi^R_r(1) \circ \phi^F_r(1)](\alpha) & = \phi^F_r(1)(\phi^F_r(1)(\alpha)) = \phi^F_r(1)(\alpha \timesf 1) \\ & = \phi^F_r(1)(\ot(\faktor{\alpha}{\fin})) = (\ot(\faktor{\alpha}{\fin})) \timesf 1 \\
& = \ot \left( \faktor{(\ot(\faktor{\alpha}{\fin}))}{\fin} \right),
\end{align*}
which is in general not isomorphic to $\ot(\faktor{\alpha}{\fin})$. 

Generalizing from our checking of Case 4, we see that repeatedly applying the map $\phi^F_r(1)$ is equivalent to iterating the finite condensation map $\fc$.

\begin{lem}\label{iterating the phi 1 map}
For any $n \in \omega$ and any ordinal $\alpha$, we have $\fc^n[\alpha] \cong \phi^F_l(1)^n(\alpha)$.
\end{lem}

\begin{proof}
For $\alpha$ an ordinal, we have 
\begin{align*}
\phi^F_r(1)(\alpha) & = \alpha \timesf 1 = \ot(\faktor{\alpha}{\fin}) \cong \faktor{\alpha}{\fin} \cong \fc[\alpha],\\
(\phi^F_r(1))^2(\alpha) & = \phi^F_r(1)(\phi^F_r(1)(\alpha)) 
 \cong \faktor{(\faktor{\alpha}{\fin})}{\fin} \cong \fc^2[\alpha],\\
\end{align*}
and, in general, for $n \in \omega$ with $n \geq 2$,
\begin{align*}
(\phi^F_r(1))^{n+1}(\alpha) & = \phi^F_r(1)[(\phi^F_r)^n(1)(\alpha)] = \phi^F_r(1)(\fc^n[\alpha]) \cong \fc^{n+1}[\alpha].
\end{align*}

\end{proof}

\section{The $\phi^F_l$ maps on ordinals of finite degree}\label{s: the phi maps on ordinals of finite degree}

To further describe the maps $\phi^F_l(\omega)$ (left multiplication, modulo the finite condensation, by $\omega$) and $\phi^F_l(1)$ (modding out by $\fin$), we use the Cantor Normal Form Theorem, which states that any ordinal can be decomposed uniquely into a certain polynomial in $\omega$. (This is Theorem 3.46 in \cite{Ro}. Note that ordinal multiplication in \cite{Ro} is written antilexicographically.) 

\begin{thm}[Cantor Normal Form]\label{Cantor normal form}
Let $\alpha$ be an ordinal. Then $\alpha$ can be written in the form 
\[
n_1 \omega^{\alpha_1} + \cdots + n_k \omega^{\alpha_k}
\]
where $\alpha_1 > \alpha_2 > \cdots > \alpha_k$ are ordinals and where $k$ and $n_1, \ldots, n_k$ are natural numbers (with $n_1 \neq 0$). Further, this decomposition is unique. \hfill \qed
\end{thm}

The largest power $\alpha_1$ of $\omega$ that appears in the Cantor normal form of an ordinal $\alpha$ is called the \textit{degree} of $\alpha$, denoted $\operatorname{deg}(\alpha)$. 
The ordinals of finite degree are those whose Cantor normal form looks like $\alpha = a_n \omega^n + a_{n-1} \omega^{n-1} + \cdots + a_1 \omega + a_0$ for some $n, a_n, \ldots, a_0 \in \omega$ with $n \neq 0$. In this section, we show in Corollary \ref{left mult by omega mod fin on ordinals of finite VD-rank} that when $\alpha$ is of finite degree, the map $\phi^F_l(\omega)$ has a very simple-to-describe effect: it sends 
$\alpha$ to $\omega^{\operatorname{deg}(\alpha)}$. To do this, we show in Proposition \ref{omega times CNF of finite degree} that if $\alpha$ is an ordinal of finite degree in Cantor normal form, then the lexicographic product $\omega \alpha$ is isomorphic to $\omega^{\operatorname{deg}(\alpha)+1}$.

\begin{lem}\label{omega left 1 reduces degree}
For any $n \in \omega$, $\faktor{\omega^{n+1}}{\fin} \cong \omega^n$.
\end{lem}

\begin{proof}
Let $n \in \omega$. By the associativity of linear order multiplication, $\omega^{n+1} \cong \omega^n \omega$. Then by Lemma \ref{right absorption mod fin in R},
\[
\faktor{\omega^{n+1}}{\fin} \cong \faktor{\omega^n \omega}{\fin} \cong \omega^n.
\]
\end{proof}

\begin{lem}\label{how powers of omega begin}
If $m < n < \omega$, then $\omega^n$ begins with $\omega$-many copies of $\omega^m$, and there is a well-ordered set $D$ such that as a linear order, $\omega^n \cong \omega \omega^m + D$.  
\end{lem}

\begin{proof}
Suppose $m < n < \omega$. By the associativity of linear order multiplication, $\omega^n \cong \omega \omega^{n-1}$, which is the linear order obtained by replacing each element of $\omega$ with a copy of $\omega^{n-1}$.

Consider the first of these copies of $\omega^{n-1}$, the one that replaced $0$ in forming $\omega \omega^{n-1}$. This copy of $\omega^{n-1}$ can be written as $\omega \omega^{n-2}$, which is obtained by replacing each element of $\omega$ with a copy of $\omega^{n-2}$.

Consider the first of these copies of $\omega^{n-2}$. This copy of $\omega^{n-2}$ can be written as $\omega \omega^{n-3}$, which is obtained by replacing each element of $\omega$ with a copy of $\omega^{n-3}$.

Continuing to break down successively smaller initial segments of $\omega^n$ in this way $(n-m)$ times, we see that $\omega^n$ begins with $\omega$-many copies of $\omega^m$, or $\omega \omega^m$. The rest of $\omega^n$, which we will call $D$, lies to the right of this initial piece. Clearly $D$ is well-ordered, since it is a subset of $\omega^n$, and we have $\omega^n \cong \omega \omega^m + D$.
\end{proof}

\begin{lem}\label{lower powers of omega coming before higher ones get combined}
If $m<n<\omega$, then $\omega^m + \omega^n \cong \omega^n$.
\end{lem}

\begin{proof}
By Lemma \ref{how powers of omega begin}, $\omega^n \cong \omega \omega^m + D$ for some well-ordered linear order $D$. Then by associativity of linear order addition,
\[
\omega^m + \omega^n \cong \omega^m + (\omega \omega^m + D) \cong (\omega^m + \omega \omega^m) + D \cong \omega \omega^m + D \cong \omega^n.
\]
The second-to-last isomorphism here is because, by re-numbering, we have $\omega^m +\omega \omega^m \cong \omega \omega^m$.
\end{proof}

\begin{lem}\label{finitely many lower powers of omega coming before higher ones get combined}
If $m<n<\omega$ and $a, b \in \omega$ with $a, b \neq 0$, then $a \omega^m + b \omega^n \cong b \omega^n$.
\end{lem}

\begin{proof}
Suppose $m<n<\omega$ and $a, b \in \omega$ with $a, b \neq 0$. Then
\begin{align*}
a \omega^m + b \omega^n & \cong \underbrace{(\omega^m + \cdots + \omega^m)}_{a \textrm{ times}} + \underbrace{(\omega^n + \cdots + \omega^n)}_{b \textrm{ times}} \\
& \cong \underbrace{(\omega^m + \cdots + \omega^m)}_{a-1 \textrm{ times}} + \omega^m + \omega^n + \underbrace{(\omega^n + \cdots + \omega^n)}_{b-1 \textrm{ times}} \\
& \cong \underbrace{(\omega^m + \cdots + \omega^m)}_{a-1 \textrm{ times}} + \omega^n + \underbrace{(\omega^n + \cdots + \omega^n)}_{b-1 \textrm{ times}} \quad \textrm{(by Lemma \ref{lower powers of omega coming before higher ones get combined})} \\
& \cong \underbrace{(\omega^m + \cdots + \omega^m)}_{a-2 \textrm{ times}} + \omega^m + \omega^n + \underbrace{(\omega^n + \cdots + \omega^n)}_{b-1 \textrm{ times}} \\
& \cong \underbrace{(\omega^m + \cdots + \omega^m)}_{a-2 \textrm{ times}} + \omega^n + \underbrace{(\omega^n + \cdots + \omega^n)}_{b-1 \textrm{ times}} \quad \textrm{(by Lemma \ref{lower powers of omega coming before higher ones get combined})} \\
& \cong \\
&  \vdots \\
& \\
& \cong (\omega^m) + \omega^n + \underbrace{(\omega^n + \cdots + \omega^n)}_{b-1 \textrm{ times}} \\
& \cong \omega^n + \underbrace{(\omega^n + \cdots + \omega^n)}_{b-1 \textrm{ times}} \quad \textrm{(by Lemma \ref{lower powers of omega coming before higher ones get combined})} \\
& \cong \underbrace{(\omega^n + \cdots + \omega^n)}_{b \textrm{ times}} \cong b \omega^n.\\
\end{align*}
\end{proof}

Note that Lemma \ref{finitely many lower powers of omega coming before higher ones get combined} implies that if $m<n<\omega$ and $a, b \in \omega$ with $a, b \neq 0$, then $b \omega^n$ begins with a copy of $a \omega^m$; that is, $a \omega^m$ is an initial segment of $b \omega^n$.

\begin{lem}\label{k alpha is an initial segment}
Let $\alpha$ be an ordinal of finite degree with Cantor normal form $\alpha = a_n \omega^n + \cdots + a_1 \omega + a_0$. Then for each $k \in \omega$, $k \alpha$ is an initial segment of $(k+1)a_n \omega^n$.
\end{lem}

\begin{proof}
Let $\alpha$ be as above, and fix $k \in \omega$. Let $T$ denote the principal tail of $\alpha$: 
\[
T:=a_{n-1}\omega^{n-1} + \cdots + a_1 \omega + a_0.
\]
By applying Lemma \ref{finitely many lower powers of omega coming before higher ones get combined} $n$ times, we get that $T + a_n \omega^n \cong a_n \omega^n$. Using this and the associativity of ordinal addition, we obtain 
\begin{align*}
(k+1) a_n \omega^n & \cong \underbrace{a_n \omega^n + \cdots + a_n \omega^n}_{k+1 \textrm{ times}} \\
& \cong a_n \omega^n + \underbrace{(T+a_n \omega^n) + \cdots + (T+a_n \omega^n)}_{k \textrm{ times}} \\
& \cong \underbrace{(a_n \omega^n + T) + \cdots + (a_n \omega^n + T)}_{k \textrm{ times}} + a_n \omega^n \\
& = \underbrace{\alpha + \cdots + \alpha}_{k \textrm{ times}} + a_n \omega^n \\
& \cong k \alpha + a_n \omega^n.\\
\end{align*}
Therefore $k \alpha$ is an initial segment of $(k+1)a_n \omega^n$.
\end{proof}

\begin{prop}\label{omega times CNF of finite degree}
Suppose $\alpha$ is an ordinal of finite degree and $\alpha \cong a_n \omega^n + \cdots + a_1 \omega + a_0$ in Cantor normal form. Then $\omega \alpha \cong \omega \omega^n \cong \omega^{n+1}$.
\end{prop}

\begin{proof}
Let $\alpha$ be as above. Then $\omega \alpha$ is also of finite degree, so we can express $\omega \alpha$ in Cantor normal form as $\omega \alpha = b_m \omega^m + \cdots + b_1 \omega + b_0$ for some $m$ and $b_m, \ldots, b_0 \in \omega$. Clearly $m \geq n$ (that is, the degree of $\omega \alpha$ is at least as big as the degree of $\alpha$). We could not have $m=n$, though: for suppose by way of contradiction that we did. Note that $k \alpha$ is isomorphic to an initial segment of $\omega \alpha$ for any $k \in \omega$. Also, since $a_n \omega^n$ embeds into $\alpha$, $k a_n \omega^n$ embeds into $k \alpha$ for any $k$. Then under the assumption that $m=n$, we would have that for any $k \in \omega$,
\[
k a_n \omega^n \preceq k (a_n \omega^n + \cdots + a_0) \cong k \alpha \preceq \omega \alpha \cong b_n \omega^n + b_{n-1} \omega^{n-1} + \cdots + b_1 \omega + b_0.
\]
However, by choosing $k$ large enough, we could make $k a_n > b_n$, so that we $k a_n \omega^n$ could not embed into $b_n \omega^n$; there would be at least one copy of $\omega^n$ left over. Also, note that no copy of $\omega^n$ could embed into the tail $b_{n-1} \omega^{n-1} + \cdots + b_1 \omega + b_0$ of $\omega \alpha$. Thus we have a contradiction, and therefore $m \geq n+1$.

Next we claim that we can embed $\omega \alpha$ into $\omega (a_n \omega^n)$. $\omega \alpha$ is formed by replacing each $j \in \omega$ with a copy of $\alpha$, and $\omega (a_n \omega^n)$ is formed by replacing each $j \in \omega$ with a copy of $a_n \omega^n$. By Lemma \ref{k alpha is an initial segment}, we can embed the first copy of $\alpha$ into the first two copies of $a_n \omega^n$; and we can embed the second copy of $\alpha$ into the third and fourth copies of $a_n \omega^n$; etc.

Thus $\omega \alpha$ embeds into $\omega (a_n \omega^n)$. Noting that $\omega (a_n \omega^n) \cong \omega \omega^n$ (just by re-numbering), we have that $\omega \alpha$ embeds into $\omega \omega^n \cong \omega^{n+1}$. We know from the above that $\omega \alpha$ does not embed into any ordinal of degree less than $n+1$. Therefore, since $\omega^{n+1}$ is the least ordinal of degree $n+1$, we have that  $\omega^{n+1}$ is the Cantor normal form of $\omega \alpha$. Thus $\omega \alpha \cong \omega \omega^n \cong \omega^{n+1}$, as desired. 
\end{proof}

\begin{cor}\label{left mult by omega mod fin on ordinals of finite VD-rank}
Let $\alpha$ be an ordinal of finite degree with Cantor normal form $\alpha = a_n \omega^n + a_{n-1} \omega^{n-1} + \cdots + a_1 \omega + a_0$. Then $\phi^F_l(\omega)(\alpha) =  \omega^n$; that is, $\phi^F_l(\omega)(\alpha) = \omega^{\operatorname{deg}(\alpha)}$. 
\end{cor}

\begin{proof}
Suppose $\alpha = a_n \omega^n + a_{n-1} \omega^{n-1} + \cdots + a_1 \omega + a_0$ is an ordinal of degree $n$ in Cantor normal form, for some $n \in \omega$. By Lemma \ref{omega times CNF of finite degree}, $\omega \alpha \cong \omega^{n+1}$. Then by Lemma \ref{omega left 1 reduces degree},
\[
\phi^F_l(\omega)(\alpha) = \omega \timesf \alpha = \ot(\faktor{\omega \alpha}{\fin}) \cong \faktor{\omega^{n+1}}{\fin} \cong \omega^n.
\]
\end{proof}

Thus  the map $\phi^F_l(\omega)$ maps any ordinal $\alpha \in \fd$ to the monic monomial $\omega^{\deg(\alpha)}$; the leading coefficient and the lower terms are lost. That is, \[\phi^F_l(\omega): \fd \to \textrm{MMonomials}(\fd).
		\]
Noting that $\rm{MMonomials}(\fd)$ can be identified with $\omega$, we have that  $\phi^F_l(\omega)$ can be identified with the degree map on $\fd$.

We next describe the effect of $\phi^F_l(1)$ on the Cantor normal form of ordinals of finite degree. (Recall that $\phi^F_r(1)$ has the same effect as $\phi^F_l(1)$, both maps send an ordinal $\alpha$ to $\ot(\faktor{\alpha}{\fin})$.)

It is listed as an exercise in \cite{Ro} that for any linear orders $A$ and $B$, $\faktor{(A + B)}{\fin}$ embeds into $\faktor{A}{\fin} + \faktor{B}{\fin}$. $\faktor{(A + B)}{\fin}$ and $\faktor{A}{\fin} + \faktor{B}{\fin}$ are not always isomorphic, though; for example, take $A=\omega^*$ and $B=\omega$. Then $\faktor{(\omega^* + \omega)}{\fin} \cong \faktor{\Z}{\fin} \cong 1$, but $\faktor{\omega^*}{\fin} + \faktor{\omega}{\fin} \cong 1+1 \cong 2$. In this example, we failed to get an isomorphism because $\omega^*$ had a last element and $\omega$ had a first element, and so the sum of the two linear orders constituted a single equivalence class mod $\fin$. Lemma \ref{when finite condensation doesn't distribute} says that this is the only situation where $\faktor{(A + B)}{\fin}$ and $\faktor{A}{\fin} + \faktor{B}{\fin}$ would not be isomorphic.

\begin{lem}\label{when finite condensation doesn't distribute}
Suppose $A$ and $B$ are linear orders such that $\faktor{(A + B)}{\fin} \not\cong \faktor{A}{\fin} + \faktor{B}{\fin}$. Then $A$ has a last element and $B$ has a first element.
\end{lem}

\begin{proof}
Suppose $\faktor{(A + B)}{\fin} \not\cong \faktor{A}{\fin} + \faktor{B}{\fin}$. This means that there is an $x \in A$ such that $\fc^A(x) \subsetneq \fc^{A+B}(x)$. (We could argue similarly if there were a $y \in B$ such that $\fc^B(y) \subsetneq \fc^{A+B}(y)$.) Choose $y \in B$ such that $y \in \fc^{A+B}(x)$. Then $y \fin x$, so there are only finitely many elements of $A+B$ between $x$ and $y$. This could not happen if $A$ had no last element, and it also could not happen if $B$ had no first element. Thus $A$ has a last element and $B$ has a first element.
\end{proof}

\begin{lem}\label{distribution of finite condensation over sums}
Let $\alpha = n_1 \omega^{\alpha_1} + \cdots + n_k \omega^{\alpha_k}$ be an ordinal in Cantor normal form. Then $\faktor{\alpha}{\fin} \cong \sum_{i=1}^k \faktor{n_i \omega^{\alpha_i}}{\fin}$.   
\end{lem}

\begin{proof}
Let $\alpha = n_1 \omega^{\alpha_1} + \cdots + n_k \omega^{\alpha_k}$ be an ordinal in Cantor normal form. We may assume that the natural numbers $n_1, \ldots, n_k$ are non-zero. The proof is by induction on $k$. 

\textit{Base step} ($k=1$): If $\alpha = n_1 \omega^{\alpha_1}$ for some $n_1 \in \omega$ ($n_1 \geq 0$) and some ordinal $\alpha_1$, then $\sum_{i=1}^k \faktor{n_i \omega^{\alpha_i}}{\fin} = \faktor{n_1 \omega^{\alpha_1}}{\fin} = \faktor{\alpha}{\fin}$. 

\textit{Induction step:} Assume the conclusion is true for sums of length $k-1$ for some $k \geq2$, and suppose that $\alpha = \sum_{i=1}^{k} \faktor{n_i \omega^{\alpha_i}}{\fin}$. By the induction hypothesis, 
\[\faktor{\left(\sum_{i=2}^k n_i \omega^{\alpha_i}\right)}{\fin} \cong \sum_{i=2}^k \left(\faktor{n_i \omega^{\alpha_i}}{\fin}\right).
\]
Since $n_1 \omega^{\alpha_1}$ has no last element (as $\alpha_1 \geq 1$), by Lemma \ref{when finite condensation doesn't distribute} and the induction hypothesis, we have 
\begin{align*}
\faktor{[n_1 \omega^{\alpha_1} + (n_2 \omega^{\alpha_2} + \cdots + n_k \omega^{\alpha_k})]}{\fin} & \cong \faktor{n_1 \omega^{\alpha_1}}{\fin} + \faktor{(n_2 \omega^{\alpha_2} + 
\cdots + n_k \omega^{\alpha_k})}{\fin} \\
& \cong \faktor{n_1 \omega^{\alpha_1}}{\fin} + \left( \faktor{n_2 \omega^{\alpha_2}}{\fin} + \cdots + \faktor{n_k \omega^{\alpha_k}}{\fin} \right) \\ &= \sum_{i=1}^k \faktor{n_i \omega^{\alpha_i}}{\fin}.\\
\end{align*}
\end{proof}

\begin{prop}\label{what finite condensation does to finite degree ordinals}
Suppose $\alpha$ is an ordinal of finite degree $n$ having the Cantor normal form $\alpha = a_n \omega^n + a_{n-1} \omega^{n-1} + \cdots + a_1 \omega + a_0$ with $n>0$. Then 
\[
\phi^F_l(1)(\alpha) = \ot(\faktor{\alpha}{\fin})
= a_n \omega^{n-1} + a_{n-1} \omega^{n-2} + \cdots + a_1 + c_\alpha\]
where $c_\alpha=0$ if $a_0=0$, and $c_\alpha=1$ if $a_0 \neq 0$.
\end{prop}

\begin{proof}
Let $\alpha$ be as above. By Corollary \ref{distribution of finite condensation over sums},
\[
\faktor{\alpha}{\fin} \cong \faktor{a_n \omega^n}{\fin} + \faktor{a_{n-1} \omega^{n-1}}{\fin} + \cdots + \faktor{a_2 \omega^2}{\fin} + \faktor{a_1 \omega}{\fin} + \faktor{a_0}{\fin}
\]
By associativity of linear order multiplication, this sum is isomorphic to 
\[
\faktor{(a_n \omega^{n-1})\omega}{\fin} + \faktor{(a_{n-1} \omega^{n-2}) \omega}{\fin} + \cdots + \faktor{(a_2 \omega) \omega}{\fin} +\faktor{a_1 \omega}{\fin} + \faktor{a_0}{\fin}
\]
-- and this sum, by Lemma \ref{right absorption mod fin in R}, is isomorphic to 
\[
a_n \omega^{n-1} + a_{n-1} \omega^{n-2} + \cdots + a_2 \omega + (a_1 + 1)
\]
if $a_0 \neq 0$, and 
\[
a_n \omega^{n-1} + a_{n-1} \omega^{n-2} + \cdots + a_2 \omega + a_1 
\]
if $a_0 = 0$.
\end{proof}

(The  condition $n>0$ in Proposition \ref{what finite condensation does to finite degree ordinals} is necessary because if $\alpha$ is of degree $0$, say $\alpha=k$ for some $k \in \omega$, then we just have $\phi^F_l(1)(\alpha)=c_\alpha$.)

Thus the map $\phi^F_l(1)$ acts on the Cantor normal form of an ordinal of finite degree by dropping the exponent $i$ on all terms of the form $a_i \omega^i$ to $i-1$, except possibly for a last nonzero constant term, which goes to $1$; and the result is another ordinal in Cantor normal form. If $\alpha$ is an ordinal of finite degree $n \geq 1$, then $\phi^F_l(1)(\alpha)$ has degree $n-1$. In this respect, the map $\phi^F_l(1)$ acts on the ordinals of finite degree similarly to how the familiar derivative operator $\frac{d}{dx}$ acts on a polynomial function $p: \mathbb{R} \to \mathbb{R}$.

The map $\phi^F_l(1)$ has another connection to a different sort of derivative: recall that for $(X, \tau)$ a topological space and $A \subseteq X$, the derived set $A^\mathrm{d}$ of $A$ consists of all of the accumulation points of $A$. ($x$ is an accumulation point of $A$ if $x \in \overline{A \setminus \{x\}}$.) If $X=\alpha$ is an ordinal of finite degree equipped with the order topology, then $X^\mathrm{d}$, with the inherited order from $\alpha$, is order-isomorphic to $\phi^F_l(1)(\alpha) \cong \fc(\alpha)$. 

Motivated by these similarities, in the next section we examine further the properties of the endomorphism $\phi^F_l(1)$ considered as a derivative operator.

\section{Finite condensation derivatives}\label{s: finite condensation derivatives}

Note that the map $\phi^F_l(1)$ is not quite the same as the finite condensation map $\fc$; whereas $\fc$ maps a linear order to a subset of itself, $\phi^F_l(1)$ is a class map on the ordinals $\ON$. Again in this section we focus on ordinals of finite degree. More precisely, we focus on the Cantor normal forms of such ordinals, which we consider as the formal sums $\sum_{i=0}^na_{n-i}\omega^i$ for $n \in \omega$ and $a_n \neq 0$. We will denote by $\fd$ the set of all Cantor normal forms of ordinals of finite degree, considered as formal sums. (We remark that $\fd$ could also be identified with a subset of $^{<\omega}\omega$, the $\omega$-branching tree of height $\omega$; for example, the ordinal $\alpha$ whose Cantor normal form is $a_n \omega^n + a_{n-1} \omega^{n-1} + \cdots + a_1 \omega + a_0$ could be encoded as the finite sequence $\langle n, a_n, a_{n-1}, \ldots, a_1, a_0 \rangle$.) Also, for the remainder of this section, we will use the notation $\bsd$ for the map $\phi^F_l(1)$; so, in this notation, 
\[
\bsd(\alpha) = \ot(1 \timesf \alpha) = \ot(\faktor{\alpha}{\fin})
\]
for $\alpha$ an ordinal. 

We can now restate Proposition \ref{what finite condensation does to finite degree ordinals} as follows: 

\begin{prop}\label{bsd is almost a linear operator on individual alpha}
For all $\alpha = a_n \omega^n + \cdots + a_1 \omega + a_0 \in \fd$ with $n>0$, 
\[
\bsd\left(\sum_{i=0}^n a_{n-i} \omega^{n-i}\right) \cong \left[\sum_{i=0}^{n-1}a_{n-i} \bsd(\omega^{n-i})\right]+c_\alpha
\cong \left[\sum_{i=0}^{n-1}a_{n-i} \omega^{n-i-1}\right]+c_\alpha \]
where $c_\alpha=0$ if $a_0=0$, and $c_0=1$ if $a_0 \neq 0$. \hfill \qed
\end{prop}

We see from Proposition \ref{bsd is almost a linear operator on individual alpha} that $\bsd$ acts almost like a polynomial derivative when applied to a sum that is the Cantor normal form of a single ordinal $\alpha \in \fd$. 

Denote by $\bsi(\alpha)$ the inverse image of $\alpha$ under $\bsd$: 
\[
\bsi(\alpha) := \{\beta \in \fd: \bsd(\beta) = \alpha\}.
\]

Suppose $\alpha \neq 0$ is of degree $0$; so $\alpha = k$ for some $k \in \omega, k \neq 0$. Then by Proposition \ref{bsd is almost a linear operator on individual alpha}, $\bsd(k \omega)=k$, so $k \omega \in \bsi(\alpha)$. If $k=1$, then $\bsd(j)=1=\alpha$ for all non-zero $j \in \omega$. If $k>1$, then $(k-1) \omega + j \in \bsi(\alpha)$ for every $j \in \omega$ with $j >0$, since $\bsd((k-1) \omega + j) = k-1+1=k=\alpha$. 

More generally, we can use Proposition \ref{bsd is almost a linear operator on individual alpha} to describe $\bsi(\alpha)$ for any $\alpha \in \fd$. Suppose $\alpha=a_n \omega^n + \cdots + a_1 \omega + a_0$ has degree at least $1$. Then $\beta := a_n \omega^{n+1} + a_{n-1}\omega^n + \cdots + a_1 \omega^2 + a_0 \omega \in \bsi(\alpha)$. If $a_0 = 1$, then $\bsi(\alpha)$ also includes any $\beta$ of the form $a_n \omega^{n+1} + a_{n-1}\omega^n + \cdots + a_1 \omega^2 + j$ for $j \in \omega, \, j >0$. If  $a_0 > 1$, then $\bsi(\alpha)$ also includes any $\beta$ of the form $a_n \omega^{n+1} + a_{n-1}\omega^n + \cdots + a_1 \omega^2 + (a_0-1)\omega + j$ for $j \in \omega, \, j >0$. Thus we have the following description of $\bsi(\alpha)$.

\begin{prop}\label{BS integrals}
Suppose $\alpha = a_n \omega^n + \cdots + a_0$ is a nonzero ordinal of finite degree. 
\begin{enumerate}
\item If $a_0=0$, then $\bsi(\alpha)$ consists of the single ordinal $a_n \omega^{n+1} + \cdots + a_1 \omega^2$.
\item If $a_0=1$, then $\bsi(\alpha)$ consists of $a_n \omega^{n+1} + \cdots + a_1 \omega^2 + \omega$ along with all ordinals of the form $a_n \omega^{n+1} + \cdots + a_1 \omega^2 + j$ for $j \in \omega, \, j > 0$.
\item If $a_0 > 1$, then $\bsi(\alpha)$ consists of $a_n \omega^{n+1} + \cdots + a_1 \omega^2 + a_0\omega$ along with all ordinals of the form $a_n \omega^{n+1} + \cdots + a_1 \omega^2 + (a_0-1)\omega + j$ for $j \in \omega, \, j > 0$.
\end{enumerate}
\hfill \qed
\end{prop}

Proposition \ref{BS integrals} implies that if $\alpha$ has no nonzero constant term (that is, $\alpha$ is not a successor ordinal), then $|\bsi(\alpha)|=1$. Thus we obtain the following corollary.

\begin{cor}\label{bsi as an inverse}
When restricted to the set of limit ordinals of finite degree, $\bsi$ defines a function; moreover, $\bsi$ is the inverse of $\bsd$. \hfill \qed
\end{cor}

We emphasize that for any $\alpha$ of finite degree, we can consider the Cantor normal form of $\alpha$ as a formal sum $\sum_{i=0}^{n}a_{n-i}\omega^{n-i}$. We define a map $\Phi$ by letting $\Phi(\alpha)$ be the Cantor normal form of $\alpha$. In particular, $\Phi$ maps the set of ordinals of finite degree to $\fd$; that is, as a class map, $\Phi: \{\alpha \in \ON: \operatorname{deg}(\alpha)< \omega\} \to \fd$ is the map sending $\alpha$ to its Cantor normal form. Conversely, given a (formal) sum $\sum_{i=0}^{n}a_{n-i}\omega^{n-i}$, we can define the inverse $\Phi^{-1}: \fd \to \ON$ as the unique ordinal whose Cantor normal form is $\sum_{i=0}^{n}a_{n-i}\omega^{n-i}$. The map $\bsd$ is a map from $\fd$ to $\fd$, as it (and its dual map $\bsi$, defined above), when applied to an ordinal in Cantor normal form, has as its output an ordinal in Cantor normal form.

Thus we have the following:
\[
\begin{tikzcd}
\fd \arrow[r, "\bsd"]
& \fd \arrow[d, "\Phi^{-1}"] \\
\{\alpha \in \ON: \operatorname{deg}(\alpha)<\omega\} \arrow[u, "\Phi"] \arrow[r, dashed, red, "\bsd^\dagger" blue]
&  \{\alpha \in \ON: \operatorname{deg}(\alpha)<\omega\}).
\end{tikzcd}
\]
Here, the derivative $\bsd^\dagger$  mapping $\{\alpha \in \ON: \operatorname{deg}(\alpha)<\omega\}$ to $\{\alpha \in \ON: \operatorname{deg}(\alpha)<\omega\}$ is an induced derivative, arising from our definition of the operator $\bsd$ defined on $\fd$. Next, consider iterating $\bsd^\dagger$. We have
\begin{align*}
\bsd^\dagger \circ \bsd^\dagger & = \bsd^\dagger \circ (\Phi^{-1} \circ \bsd \circ \Phi) \\
& = (\Phi^{-1} \circ \bsd \circ \Phi) \circ (\Phi^{-1} \circ \bsd \circ \Phi) \\
& = \Phi^{-1} \circ \bsd \circ \rm{Id} \circ \bsd \circ \Phi \\
& = \Phi^{-1} \circ \bsd^2 \circ \Phi.
\end{align*}
Similarly, for $n<\omega$, we will have $(\bsd^\dagger)^n = \Phi^{-1} \circ \bsd^n \circ \Phi$.
More generally, suppose we have defined an endomorphism $f \in \operatorname{End}(\fd)$. We can then define $\operatorname{Ad}_\Phi(f) \in \operatorname{End}(\{\alpha \in \ON: \operatorname{deg}(\alpha)<\omega\})$ by $\operatorname{Ad}_\Phi = \Phi^{-1} \circ f \circ \Phi$. 

We conclude by examining the extent to which the finite condensation derivative $\bsd$ acts like a linear map. In general, for $\alpha, \beta \in \fd$ and $p, q \in \omega$, we do not have $\bsd(\Phi(p\alpha + q\beta)) \cong p\bsd(\alpha) + q\bsd(\beta)$; however, we show that these differ only by a constant, and they are isomorphic in the case when $\alpha$ and $\beta$ are limit ordinals.

\begin{remark}
We have only defined $\bsd$ on Cantor normal forms, and the expression $\alpha + \beta = a_n \omega^n + \cdots + a_0 + b^m + \cdots + b_0$ need not be in Cantor normal form. For ease of notation, we will write ``$\bsd(p\alpha)$'' and ``$\bsd(\alpha + \beta)$'' for $\bsd(\Phi(p\alpha))$ and $\bsd(\Phi(\alpha + \beta))$ respectively; it will be understood that the argument of $\bsd$ in such expressions is put into Cantor normal form before the finite condensation derivative is taken. For similar reasons, we will write ``$\bsd(\alpha) + \bsd(\beta)$'' to mean $\Phi(\bsd(\alpha) + \bsd(\beta)).$
\end{remark}

\begin{lem}\label{scalar times alpha formula}
Let $\alpha=a_n \omega^n + \cdots + a_0$ be an ordinal of finite degree, and let $p \in \omega$ with $p \geq 1$. Then $p \alpha \cong (p-1)a_n \omega^n + \alpha \cong p a_n \omega^n + a_{n-1} \omega^{n-1} + \cdots + a_0$.
\end{lem}

\begin{proof}
If $p=1$, then for any $\alpha$, $p \alpha = \alpha = 0 + \alpha = (p-1)a_n \omega^n + \alpha$. 

Suppose $p \geq 2$ and $n=0$. If $\alpha=0$, then $p \alpha \cong (p-1)a_n \omega^n + \alpha = 0$. If $\alpha = a_0 \neq 0$, then $p \alpha = pa_0 = (p-1)a_0 + a_0 = (p-a)a_0 + \alpha$.

Now suppose $p \geq 2$ and $n \geq 1$. Denote by $T_\alpha$ the principal tail of $\alpha$; that is, $T_\alpha = a_{n-1} \omega^{n-1} + \cdots + a_0$. Note that each term in $T_\alpha$ has degree less than $n$, so that $T_\alpha + a_n \omega^n \cong a_n \omega^n$ by Lemma \ref{finitely many lower powers of omega coming before higher ones get combined}. Then
\begin{align*}
p \alpha & \cong p (a_n \omega^n + T_\alpha) \cong \underbrace{(a_n \omega^n + T_\alpha) + \cdots + (a_n \omega^n + T_\alpha)}_{p \textrm{ times}} \\
& \cong a_n \omega^n + \underbrace{(T_\alpha + a_n \omega^n) + \cdots + (T_\alpha + a_n \omega^n)}_{p-1 \textrm{ times}} +T_\alpha\\
& \cong a_n \omega^n + \underbrace{(a_n \omega^n + \cdots + a_n \omega^n)}_{p-1 \textrm{ times}} + T_\alpha \\
& \cong \underbrace{(a_n \omega^n + \cdots + a_n \omega^n)}_{p-1 \textrm{ times}} + (a_n \omega^n + T_\alpha)\\
& \cong (p-1)a_n \omega^n + \alpha, 
\end{align*}
and  $(p-1)a_n \omega^n + \alpha \cong p a_n \omega^n + T_\alpha \cong p a_n \omega^n + a_{n-1} \omega^{n-1} + \cdots + a_0$.
\end{proof}

From Lemma \ref{scalar times alpha formula}, we get the following corollary.

\begin{cor}\label{scalar times derivative of alpha}
Let $\alpha = a_n \omega^n + \cdots + a_0$ be an ordinal of finite degree $n \geq 1$, and let $p \in \omega$ with $p>0$.
\begin{enumerate}
    \item If $\alpha$ has degree $n \geq 2$, then $\bsd(p \alpha) \cong p \bsd(\alpha) \cong (p-1)a_n \omega^{n-1} + \bsd(\alpha)$.
    \item If $\alpha$ has degree $1$, then $\bsd(p \alpha) \cong (p-1)a_n \omega^{n-1} + \bsd(\alpha)$, and  $p \bsd(\alpha)$ differs from $(p-1)a_n \omega^{n-1} + \bsd(\alpha)$ by at most a constant.
\end{enumerate}
\end{cor}

\begin{proof}
First suppose $\alpha$ has degree $n \geq 2$. By Proposition \ref{bsd is almost a linear operator on individual alpha},  $\bsd(\alpha) = a_n \omega^{n-1} + a_{n-1} \omega^{n-2} + \cdots + a_2 \omega + a_1 + c_\alpha$ (where $c_\alpha= 0$ if $a_0=0$ and $c_{\alpha} = 1$ if $a_0 >0$). Then, letting $T_{\bsd(\alpha)}$ denote the principal tail $a_{n-1} \omega^{n-2} + \cdots + a_1 + c_\alpha$ of $\alpha$, we have that every term in $T_{\bsd(\alpha)}$ has degree strictly smaller than the degree of $\bsd(\alpha)$, which is $n-1$. Thus
\begin{align*}
p \bsd(\alpha) & \cong p (a_n \omega^{n-1} + T_{\bsd(\alpha)}) \\
& \cong \underbrace{(a_n \omega^{n-1} + T_{\bsd(\alpha)}) + \cdots + (a_n \omega^{n-1} + T_{\bsd(\alpha)})}_{p \textrm{ times}} \\
& \cong \underbrace{a_n \omega^{n-1} + \cdots + a_n \omega^{n-1}}_{p \textrm{ times}} + T_{\bsd(\alpha)} \quad \textrm{(by Lemma \ref{finitely many lower powers of omega coming before higher ones get combined})}\\
& \cong \underbrace{a_n \omega^{n-1} + \cdots + a_n \omega^{n-1}}_{p-1 \textrm{ times}} + (a_n \omega^{n-1} + T_{\bsd(\alpha)}) \\
& \cong (p-1)a_n \omega^{n-1} + \bsd(\alpha).
\end{align*}
Also,
\begin{align*}
\bsd(p \alpha) & \cong \bsd(p a_n \omega^n + a_{n-1} \omega^{n-1} + \cdots + a_1 \omega + a_0) \quad \textrm{(by Lemma \ref{scalar times alpha formula})} \\
& \cong p a_n \omega^{n-1} + a_{n-1}\omega^{n-2} + \cdots + a_1 + c_\alpha \quad \textrm{(by Proposition \ref{bsd is almost a linear operator on individual alpha})} \\
& \cong (p-1)a_n \omega^{n-1} + a_n \omega^{n-1}+ a_{n-1}\omega^{n-2} + \cdots + a_1 + c_\alpha \\
& \cong (p-1)a_n \omega^{n-1} + \bsd(\alpha).
\end{align*}

Next, suppose $\alpha$ has degree $1$; say $\alpha = a_1 \omega + a_0$ for some $a_1, a_0 \in \omega$ with $a_1 \neq 0$. Then $\bsd(\alpha) = a_1 + c_\alpha$ where $c_\alpha = 0$ if $a_0=0$ and $c_\alpha = 1$ if $a_0 > 1$. Then $p \bsd(\alpha) \cong p(a_1 + c_\alpha) \cong p a_1 + p c_\alpha$, but $(p-1)a_n \omega^{n-1} + \bsd(\alpha) \cong (p-1)a_1 + (a_1 + c_\alpha) \cong pa_1 + c_\alpha$. This means that in this case $p \bsd(\alpha)$ is not isomorphic to $(p-1)a_n \omega^{n-1} + \bsd(\alpha)$ unless the constant term $a_0$ is zero. However, since $a_0 + a_1 \omega \cong a_1 \omega$,
\[
\bsd(p \alpha) \cong \bsd(p(a_1 \omega + a_0)) \cong \bsd(p a_1 \omega + a_0) \cong p a_1 + c_\alpha,
\]
so that $\bsd(p \alpha)$ is isomorphic to $(p-1)a_n \omega^{n-1}+ \bsd(\alpha)$ in this case.
\end{proof}

\begin{prop}\label{when bsd distributes over addition}
Let $\alpha$ and $\beta$ be nonzero ordinals of finite degree in Cantor normal form. Then we have $\bsd(\alpha + \beta) \cong \bsd(\alpha) + \bsd(\beta)$ in the following cases:
\begin{enumerate}
    \item $\alpha$ is a limit ordinal; or
    \item $\alpha$ is a successor ordinal and $\operatorname{deg}\beta \geq 2$.
\end{enumerate}
In all other cases, we have $\bsd(\alpha + \beta) +1 \cong \bsd(\alpha) + \bsd(\beta).$
\end{prop}

\begin{proof}
Let $\alpha, \beta \in \fd$ be nonzero. If $\alpha$ is a limit ordinal, then $\bsd(\alpha + \beta) \cong \bsd(\alpha) + \bsd(\beta)$ by Lemma \ref{when finite condensation doesn't distribute}. 

Suppose $\alpha$ is a successor ordinal; say $\alpha = a_n \omega^n + \cdots + a_0$ where $a_0 >0$. We consider cases depending on $\operatorname{deg}(\alpha)$ and $\operatorname{deg}(\beta)$. The first three cases are for when $\deg(\beta)$ is at least $2$.

\subparagraph{\textit{Case 1:}}
Suppose $\deg(\beta) \geq 2$ and  $\deg(\alpha) < \deg(\beta)$. Since $\beta$ has degree at least $2$, we have $\deg(\bsd(\alpha)) < \deg(\deg(\beta))$ and $\deg(\bsd(\beta)) \geq 1$. Then by Lemma \ref{lower powers of omega coming before higher ones get combined}, 
\begin{equation}\label{formula case 1 bsd sum}
\bsd(\alpha + \beta) \cong \bsd(\beta) \cong \bsd(\alpha) + \bsd(\beta).
\end{equation}

\subparagraph{\textit{Case 2:}} Suppose $2 \leq \deg(\alpha) = \deg(\beta)$. Then by Lemma \ref{lower powers of omega coming before higher ones get combined},
\begin{equation}\label{formula case 2 bsd sum}
\begin{split}
\bsd(\alpha + \beta) & \cong \bsd(a_n \omega^n + \cdots + a_0 + b_n \omega^n + \cdots + b_0) \\
& \cong \bsd((a_n+b_n)\omega^n + b_{n-1}\omega^{n-1} + \cdots + b_0) \\
& \cong (a_n + b_n)\omega^{n-1} + b_{n-1}\omega^{n-2} + \cdots + b_1 + c_\beta,
\end{split}
\end{equation}
and 
\begin{align*}
\bsd(\alpha) + \bsd(\beta) & \cong (a_n \omega^{n-1} + \cdots + a_1 + c_\alpha) + (b_n \omega^{n-1} + \cdots + b_1 + c_\beta) \\ 
& = (a_n + b_n) \omega^{n-1} + b_{n-1}\omega^{n-2} + \cdots + b_1 + c_\beta
\end{align*}
as well.

\subparagraph{\textit{Case 3:}} Suppose $2 \leq \deg(\beta) < \deg(\alpha)$. Then by Lemma \ref{lower powers of omega coming before higher ones get combined},
\begin{equation}\label{formula case 3 bsd sum}
\begin{split}
\bsd(\alpha + \beta) & \cong \bsd(a_n \omega^n + \cdots + a_m \omega^m + \cdots + a_0 + b_m \omega^m + \cdots + b_0) \\
& \cong \bsd(a_n \omega^n + \cdots + (a_m + b_m)\omega^m + b_{m-1}\omega^{m-1}+ \cdots + b_0) \\
& \cong a_n \omega^{n-1} + \cdots + (a_m + b_m)\omega^{m-1} + b_{m-1}\omega^{m-2}+ \cdots + b_1 + c_\beta, 
\end{split}
\end{equation}
and also 
\begin{align*}
\bsd(\alpha) + \bsd(\beta) & \cong a_n \omega^{n-1} + \cdots + a_m \omega^{m-1} + \cdots + a_1 + c_\alpha + b_m \omega^{m-1} + \cdots + b_1 + c_\beta \\
& \cong a_n \omega^{n-1} + \cdots + (a_m + b_m) \omega^{m-1} + b_{m-1} \omega^{m-2} + \cdots + b_1 + c_\beta
\end{align*}
as well.

Thus in all cases where $\alpha$ is a successor ordinal and  $\deg(\beta) \geq 2$, we have $\bsd(\alpha + \beta) \cong \bsd(\alpha) + \bsd(\beta)$.

The next two cases are for when $\alpha$ is a successor ordinal of degree at least $2$ and $\deg(\beta) \leq 1$.

\subparagraph{\textit{Case 4:}} Suppose that $\alpha$ is a successor ordinal of degree at least $2$ and $\deg(\beta) = 1$; say $\beta = b_1 \omega + b_0$ for some $b_1 > 0$. Then by Lemma \ref{lower powers of omega coming before higher ones get combined},
\begin{align*}
\bsd(\alpha + \beta) & \cong \bsd(a_n \omega^n + \cdots + a_1 \omega + a_0 + b_1 \omega + b_0) \\
& \cong \bsd(a_n \omega^n + \cdots + (a_1+b_1) \omega + b_0) \\
& \cong a_n \omega^{n-1} + \cdots + (a_1 + b_1) + c_\beta,
\end{align*}
but
\begin{align*}
\bsd(\alpha) + \bsd(\beta) & \cong \bsd(a_n \omega^n + \cdots + a_1 \omega + a_0) + \bsd(b_1 \omega + b_0) \\
& \cong a_n \omega^{n-1} + \cdots + a_1 + 1 + b_1 + c_\beta \quad \textrm{since ($a_0 > 0$)}. \\
\end{align*}
\subparagraph{\textit{Case 5:}} Suppose that $\alpha$ is a successor ordinal of degree at least $2$ and $\deg(\beta) = 0$; say $\beta = b_0$ for some $b_0 > 0$. Then since $a_0>0$, 
\begin{align*}
\bsd(\alpha + \beta) & \cong \bsd(a_n \omega^n + \cdots + (a_0 + b_0)) \\
& \cong a_n \omega^{n-1}+ \cdots + a_1 + 1,
\end{align*}
but
\begin{align*}
\bsd(\alpha) + \bsd(\beta) & \cong (a_n \omega^{n-1} + \cdots + a_1 + 1) + 1.   
\end{align*}
Finally, the last four cases are for when $\alpha$ is a successor ordinal and  both $\alpha$ and $\beta$ have degree less than $2$.

\subparagraph{\textit{Case 6:}} If both $\alpha$ and $\beta$ have degree $0$, say $\alpha = a_0 > 0$ and $\beta = b_0 > 0$, then $\bsd(\alpha + \beta) \cong 1$ but $\bsd(\alpha) + \bsd(\beta) \cong 2.$

\subparagraph{\textit{Case 7:}} If $\alpha = a_0 > 0$ and $\beta$ has degree $1$, say $\alpha = a_0 > 0$ and $\beta = b_1 \omega + b_0$ for some $\beta_1 > 0$, then 
\[
\bsd(\alpha + \beta) \cong \bsd(a_0 + b_1 \omega + b_0) \cong \bsd(b_1 \omega + b_0) \cong b_1 + c_\beta,
\]
but 
\[
\bsd(\alpha) + \bsd(\beta) \cong \bsd(a_0) + \bsd(b_1 \omega + b_0) \cong 1 + b_1 + c_\beta.
\]

\subparagraph{\textit{Case 8:}}
If $\alpha$ is a successor ordinal of degree $1$ and $\deg(\beta) = 0$, say $\alpha = a_1 \omega + a_0$ for some $a_0, a_1 > 0$ and $\beta = b_0$ for some $b_0 > 0$, then
\[
\bsd(\alpha + \beta) \cong \bsd(a_1 \omega + a_0 + b_0) \cong a_1 + 1,
\]
but
\[
\bsd(\alpha) + \bsd(\beta) \cong \bsd(a_1 \omega + a_0) + \bsd(b_0) \cong a_1 + 1 + 1.
\]

\subparagraph{\textit{Case 9:}}
Finally, if $\alpha$ is a successor ordinal and both $\alpha$ and $\beta$ have degree $1$, say $\alpha = a_1 \omega + a_0$ for some $a_0, a_1 > 0$ and $\beta = b_1 \omega + b_0$ for some $b_1 > 0$, then 
\[
\bsd(\alpha + \beta) \cong \bsd(a_1 \omega + a_0 + b_1 \omega + b_0) \cong \bsd((a_1 + b_1) \omega + b_0) \cong a_1 + b_1 + c_\beta,
\]
but
\[
\bsd(\alpha) + \bsd(\beta) \cong \bsd(a_1 \omega + a_0) + \bsd(b_1 \omega + b_0) \cong a_1 + 1 + b_1 + c_\beta.
\]
\end{proof}

If $\alpha, \beta, \gamma \in \fd$ are limit ordinals, then by applying Proposition \ref{when bsd distributes over addition} twice, we have $\bsd(\alpha + \beta + \gamma) \cong \bsd(\alpha + \beta) + \bsd(\gamma) \cong \bsd(\alpha) + \bsd(\beta) + \bsd(\gamma)$. More generally, we obtain the following corollary, arguing by induction:

\begin{cor}\label{bsd distributes over finite sums of limits}
Suppose, for $1 \leq i \leq t$, that $\alpha_i$ is a limit ordinal of finite degree in Cantor normal form. Then
\[
\bsd\left(\sum_{i=1}^t \alpha_i \right) \cong \sum_{i=1}^t \bsd(\alpha_i).
\] \hfill \qed
\end{cor}

\begin{thm}\label{bsd is linear on limit ordinals of finite degree at least 2}
Let $\alpha=a_n \omega^n + \cdots + a_1 \omega$ and $\beta = b_m \omega^m + \cdots + b_1 \omega$ be limit ordinals of finite degree at least $2$ in Cantor normal form, and let $p, q \in \omega$ with $p, q > 0$. Then $\bsd(p \alpha + q \beta) = p \bsd(\alpha) + q \bsd(\beta)$.

Moreover, we have the following expressions for $\bsd(p \alpha + q \beta)$ depending on the relationship between $\deg(\alpha)$ and $\deg(\beta)$:
\begin{enumerate}
    \item If $\deg(\alpha) < \deg(\beta)$, then \[\bsd(p \alpha + q \beta) \cong q b_m \omega^{m-1} + b_{m-1} \omega^{m-2} + \cdots + b_2 \omega + b_1;\]
    \item If $\deg(\alpha) = \deg(\beta)$, then \[\bsd(p \alpha + q \beta) \cong (p a_n+q b_n)\omega^{n-1} + b_{n-1} \omega^{n-2} + \cdots + b_2 \omega + b_1;\]
    \item If $\deg(\alpha) > \deg(\beta)$, then 
    \begin{align*}
    & \bsd(p \alpha + q \beta)   \\
    &\quad \cong p a_n \omega^{n-1} + a_{n-1} \omega^{n-2} + \cdots + a_{m+1}\omega^m + (a_m + qb_m)\omega^{m-1} + b_{m-1} \omega^{m-2} + \cdots + b_1.
    \end{align*}
\end{enumerate}
\end{thm}

\begin{proof}
Let $\alpha, \beta, p$, and $q$ be as above. Since $\alpha$ and $\beta$ are limit ordinals, so are $p \alpha$ and $q \beta$. Then by Corollary \ref{bsd distributes over finite sums of limits},
\begin{align*}
\bsd(p \alpha + q \beta) & \cong \bsd(p \alpha) + \bsd(q \beta) \\
& \cong \bsd\left(\sum_{i=1}^p \alpha\right) + \bsd\left(\sum_{i=1}^q \beta\right) \\
& \cong \sum_{i=1}^p \bsd(\alpha) + \sum_{i=1}^q \bsd(\beta) \\
& \cong p \bsd(\alpha) + q \bsd(\beta).
\end{align*}
Suppose $\deg(\alpha) < \deg(\beta)$. Then by Lemma \ref{finitely many lower powers of omega coming before higher ones get combined} and Corollary \ref{scalar times derivative of alpha},
\begin{align*}
\bsd(p \alpha + q \beta) & \cong \bsd(q \beta) \cong (q-1)b_m \omega^{m-1} + \bsd(\beta) \\
& \cong (q-1)b_m \omega^{m-1} + b_m \omega^{m-1} + b_{m-1}\omega^{m-2} + \cdots + b_2 \omega + b_1 \\
& \cong qb_m \omega^{m-1} + b_{m-1}\omega^{m-2} + \cdots + b_2 \omega + b_1.
\end{align*}
Suppose $\deg(\alpha) = \deg(\beta)$. Then by Lemma \ref{finitely many lower powers of omega coming before higher ones get combined} and Corollary \ref{scalar times derivative of alpha},
\begin{align*}
\bsd(p \alpha + q \beta) & \cong \bsd(p \alpha) + \bsd(q \beta) \\
& \cong (p-1)a_n \omega^{n-1} + \bsd(\alpha) + (q-1)b_n \omega^{n-1} + \bsd(\beta) \\
& \cong (p-1) a_n \omega^{n-1} + a_n \omega^{n-1} + \cdots + a_1 \\
& \quad \quad + (q-1)b_n \omega^{n-1} + b_n \omega^{n-1} + \cdots + b_1 \\
& \cong p a_n \omega^{n-1} + a_{n-1} \omega^{n-2} + \cdots + a_1 \\
& \quad \quad + q b_n \omega^{n-1} + b_{n-1} \omega^{n-2} + \cdots + b_1 \\
& \cong (p a_n + q b_n) \omega^{n-1} + b_{n-1} \omega^{n-2} + \cdots + b_1.
\end{align*}

Finally, suppose $\deg(\alpha) > \deg(\beta)$. Then again by Lemma \ref{finitely many lower powers of omega coming before higher ones get combined} and Corollary \ref{scalar times derivative of alpha},
\begin{align*}
\bsd(p \alpha + q \beta) & \cong \bsd(p \alpha) + \bsd(q \beta) \\
& \cong (p-1)a_n \omega^{n-1} + \bsd(\alpha) + (q-1)b_m \omega^{m-1} + \bsd(\beta) \\
& \cong (p-1)a_n \omega^{n-1} + a_n \omega^{n-1} + \cdots + a_m\omega^{m-1} + \cdots + a_1 \\
& \quad \quad + (q-1) b_m \omega^{m-1} + b_m \omega^{m-1} + \cdots + b_1 \\
& \cong p a_n \omega^{n-1} + a_{n-1} \omega^{n-2} + \cdots + a_m \omega^{m-1} + \cdots + a_1 \\
& \quad \quad + q b_m \omega^{m-1} + b_{m-1} \omega^{m-2} + \cdots + b_1 \\
& \cong p a_n \omega^{n-1} + a_{n-1} \omega^{n-2} + \cdots + a_m \omega^{m-1} + q b_m \omega^{m-1} \\
& \quad \quad + b_{m-1} \omega^{m-2} + \cdots + b_1 \\
& \cong p a_n \omega^{n-1} + a_{n-1} \omega^{n-2} + \cdots + (a_m + q b_m) \omega^{m-1} + b_{m-1} \omega^{m-2} + \cdots + b_1.
\end{align*}
\end{proof}

More generally, we obtain the following expression for the finite condensation derivative of expressions of the form $\sum_{i=1}^s p_i \alpha_i$, with some assumptions on the ordinals $\alpha_i$. 

\begin{thm}\label{lol}
Suppose, for some $s \in \omega, s \geq 2$ that $\alpha_1, \ldots, \alpha_s$ are limit ordinals of finite degree in Cantor normal form, and $p_1, \ldots, p_s$ are nonzero natural numbers. Suppose also that for each $1 \leq i \leq s$, $\alpha_i$ has leading term $a_{n_i} \omega^{n_i}$. Then
\begin{equation}\label{lol equation 1}
\bsd \left(  \sum_{i=1}^s p_i \alpha_i \right) \cong \sum_{i=1}^s \left((p_i - 1)a_{n_i}\omega^{\deg(a_i)-1} + \bsd(\alpha_i) \right).
\end{equation}
If, in addition, each $\alpha_i$ is of degree at least $2$, then  
\begin{equation}\label{lol equation 2}
\bsd \left(  \sum_{i=1}^s p_i \alpha_i \right) \cong \sum_{i=1}^s \left((p_i - 1)a_{n_i}\omega^{\deg(a_i)-1} + \bsd(\alpha_i) \right) \cong \sum_{i=1}^s p_i \bsd(\alpha_i).
\end{equation} 
\end{thm}

\begin{proof}
If each $\alpha_i$ is a limit ordinal, then each $p_i \alpha_i$ is also a limit ordinal, so we have by Corollary \ref{bsd distributes over finite sums of limits} and the first part of Corollary \ref{scalar times derivative of alpha} that 
\begin{equation}
\bsd \left(  \sum_{i=1}^s p_i \alpha_i \right) \cong \sum_{i=1}^s \bsd(p_i \alpha_i) \cong \sum_{i=1}^s \left((p_i - 1)a_{n_i}\omega^{\deg(a_i)-1} + \bsd(\alpha_i) \right).
\end{equation}
If we also assume that each $\alpha$ is of degree at least $2$, then we have by the second part of Corollary \ref{scalar times derivative of alpha} that both of these expressions are isomorphic to $\sum_{i=1}^s p_i \bsd(\alpha_i)$.
\end{proof}

Note that Proposition \ref{lol} shows only a quasi-linear behavior of $\bsd$ on $\fd$: first, we require additional properties of the ordinals $\alpha_i$; and second, $\bsd(\alpha_i)$ need not have those properties. (For example, $\omega^2 + \omega$ is a limit ordinal of degree $2$, but $\bsd(\omega^2 + \omega) \cong \omega + 1$ is a successor ordinal of degree $1$.)

One could also consider right multiplication by natural number scalars modulo the finite condensation, but the next proposition shows that such multiplication has the same effect as the finite condensation derivative $\bsd$.

\begin{prop}\label{right multiplication by scalars mod finite does nothing}
For all $\alpha \in \fd$ and $p \in \omega$ with $p>0$, $\bsd(\alpha p) \cong \bsd(\alpha)$.
\end{prop}

\begin{proof}
Let $\alpha \cong a_n \omega^n + \cdots + a_1 \omega + a_0 \in \fd$, and let $p \in \omega$ with $p>0$. Observe that $\omega p \cong \omega$ (as $\omega p$ is formed by replacing each element of $\omega$ with  a copy of $p$); and, in general, if $i>0$, then $\omega^i p \cong \omega^i$. Then by right distributivity,
\begin{align*}
\bsd(\alpha p) & \cong \bsd((a_n \omega^n + \cdots + a_1 \omega + a_0)p) \\
& \cong \bsd(a_n \omega^n p + \cdots + a_1 \omega p + a_0 p) \\
& \cong \bsd(a_n \omega^n + \cdots + a_1 \omega + a_0 p) \\
& \cong a_n \omega^{n-1} + \cdots + a_1 + c
\end{align*}
where $c=0$ if $a_0 p =0$, and $c=1$ if $a_0 p >0$; that is, since $p>0$, $c=0$ if $a_0=0$, and $c=1$ if $a_0 >0$. Therefore
\[
\bsd(\alpha p) \cong a_n \omega^{n-1} + \cdots + a_1 + c_\alpha \cong \bsd(\alpha).
\]
\end{proof}

\section{Future work}
We plan to address the following questions.
\begin{enumerate}
    \item Which of the properties of $\bsd$ acting on ordinals of finite degree remain true if we consider $\bsd$ acting on countable ordinals of degree at least $\omega$?
    \item If the finite condensation is replaced by another condensation, and if we define a multiplication of linear orders based on that condensation, what algebraic structures arise?
    \item Given a condensation $\sim$, can we characterize the set of linear orders $L$ such that $\faktor{L}{\sim} \cong 1$? 
\end{enumerate}

\section{Statements and Declarations}

The authors have no competing interests to declare.

\end{document}